\pdfoutput=1
\documentclass[11pt]{amsart}

\usepackage{microtype}
\usepackage{hyperref}
\usepackage[alphabetic]{amsrefs}
\usepackage{color}

\DeclareSymbolFont{bbold}{U}{bbold}{m}{n}
\DeclareSymbolFontAlphabet{\mathbbold}{bbold}

\newtheorem{thm}{Theorem}
\newtheorem{lem}[thm]{Lemma}
\newtheorem{prop}[thm]{Proposition}

\theoremstyle{definition}

\newtheorem{rmk}[thm]{Remark}
\newtheorem{ex}[thm]{Example}

\newcommand\arxiv[1]{\href{http://arxiv.org/abs/#1}{arXiv:#1}}
\newcommand\bfu{\mathbf u}
\newcommand\defi[1]{\emph{#1}}
\newcommand\onevector{\mathbbold 1}
\newcommand\rank{\operatorname{rank}}
\newcommand\sep{\operatorname{sep}}
\newcommand\supp{\operatorname{supp}}
\newcommand\QQ{\mathbb{Q}}
\newcommand\ZZ{\mathbb{Z}}
\newcommand\val{\operatorname{val}}
\newcommand\vecC{\mathbf C}
\newcommand\vecCco{\vecC^{\operatorname{co}}}

\title[Construction of the Lindstr\"om valuation]{Construction of the
Lindstr\"om valuation \\ of an algebraic extension}
\author{Dustin Cartwright}
\address{Department of Mathematics \\ University of Tennessee \\
         227 Ayres Hall \\ Knoxville, TN 37996}
\email{cartwright@utk.edu}

\begin{document}

\begin{abstract}
Recently, Bollen, Draisma, and Pendavingh have introduced the Lindstr\"om
valuation on the algebraic matroid of a field extension of characteristic~$p$.
Their construction passes through what they call a matroid flock and builds on
some of the associated theory of matroid flocks which they develop. In this
paper, we give a direct construction of the Lindstr\"om valuated matroid
using the theory of inseparable field extensions. In particular, we
give a description of the valuation, the valuated circuits, and the valuated
cocircuits.
\end{abstract}

\maketitle

The algebraic matroid of a field extension records which subsets of a fixed set
of elements of the field are algebraically independent. In characteristic~0, the
algebraic matroid coincides with the linear matroid of the vector configuration
of differentials, and, as a consequence the class of matroids with algebraic
realizations over a field of characteristic 0 is exactly equivalent to the class
of matroids with linear realizations in characteristic~0~\cite{ingleton}.
However, in positive characteristic, there are strictly more algebraic matroids
than linear matroids, and without an equivalence to linear matroids, the class
of algebraic matroids is not well understood.

Pioneering work of Lindstr\"om has shown the power of first applying well-chosen
powers of the Frobenius morphism to the field elements, before taking
differentials. In particular, he constructed an infinite family of matroids
(the Fano matroid among them) for which any algebraic realization over a
field of finite characteristic, after applying appropriate powers of Frobenius
and taking differentials, yields a linear representation of the same
matroid~\cite{lindstrom}.

In general, no single choice of powers of Frobenius may capture the full
algebraic matroid, and so Bollen, Draisma, and Pendavingh went one step further
by looking at the matroids of differentials after all possible powers of
Frobenius applied to the chosen field elements~\cite{bollen-draisma-pendavingh}.
These matroids fit together to form what they call a \defi{matroid flock}, and
they show that a matroid flock is equivalent to a valuated
matroid~\cite{bollen-draisma-pendavingh}*{Thm.~7}. Therefore, the matroid flock
of differentials defines a valuation on the algebraic matroid of the field
extension, called the \defi{Lindstr\"om valuation} of the algebraic matroid. In
this paper we give a direct construction of this valuation, without reference to
matroid flocks.

We now explain the construction of the Lindstr\"om valuation of an algebraic
matroid. Throughout this paper, we will work with an extension of fields $L
\supset K$ of characteristic $p>0$ as well as fixed elements $x_1, \ldots, x_n
\in L$. We also assume that $L$ is a finite extension of $K(x_1, \ldots, x_n)$,
for example, by replacing $L$ with $K(x_1, \ldots, x_n)$. The algebraic matroid
of this extension can be described in terms of its bases, which are subsets $B
\subset E = \{1, \ldots, n\}$ such that the extension of~$L$ over $K(x_B) =
K(x_i : i \in B)$ is algebraic. We recall from~\cite{lang}*{Sec.~V.6} that if
$K(x_B)^{\sep}$ denotes the set of elements of~$L$ which are separable over
$K(x_B)$, then $L$ is a purely inseparable extension of $K(x_B)^{\sep}$, and the
degree of this extension, $[L : K(x_B)^{\sep}]$ is called the \defi{inseparable
degree} and denoted by adding a subscript: $[L : K(x_B)]_i$.

Now, we define a valuation on the algebraic matroid of $L$ as the following
function $\nu$ from the set of bases to $\ZZ$:
\begin{equation}\label{eq:valuation}
\nu(B) = \log_p [L : K(x_B)]_i.
\end{equation}
Note that $\nu(B)$ is finite because we assumed that $L$ was a finitely generated
algebraic extension of $K(x_B)$ and it is an integer because $[L : K(x_B)]_i$ is
the degree of a purely inseparable extension, which is always a power
of~$p$~\cite{lang}*{Cor.~V.6.2}.

\begin{thm}\label{t:agree}
The function $\nu$ in \eqref{eq:valuation} defines a valuation on the algebraic
matroid of $L \supset K$, such that the associated matroid flock is the matroid
flock of the extension.
\end{thm}

In addition to the valuation given in~\eqref{eq:valuation}, we give
descriptions of the valuated circuits of the Lindstr\"om
valuated matroid in the beginning of Section~\ref{sec:valuated-matroid} and of
the valuated cocircuits and minors in Section~\ref{sec:cocircuits-minors}. The
description of the
circuits gives an algorithm for computing the Lindstr\"om valuated matroid using
Gr\"obner bases, assuming that $L$ is finitely generated over a prime field
(see Remark~\ref{r:computing} for details).

\begin{rmk}\label{r:sign-convention}
There are two different sign conventions used in the literature on valuated
matroids. We use the convention which is compatible with the ``min-plus''
convention in tropical geometry, which is the opposite of what was used in the
original paper of Dress and Wenzel~\cite{dress-wenzel}, but is consistent
with~\cite{bollen-draisma-pendavingh}.
\end{rmk}

\subsection*{Acknowledgments}
I'd like to thank Jan Draisma for useful discussion about the results
in~\cite{bollen-draisma-pendavingh}, which prompted this paper, Rudi Pendavingh
for suggesting the results appearing in
Section~\ref{sec:cocircuits-minors}, and Felipe Rinc\'on for helpful feedback.
The author was supported by NSA Young
Investigator grant H98230-16-1-0019.

\section{The Lindstr\"om valuated matroid}\label{sec:valuated-matroid}

In this section, we verify that the function~\eqref{eq:valuation} from the
introduction is a valuation on the algebraic matroid of the extension $L \supset
K$ and the elements $x_1, \ldots, x_n$. We do this by first constructing the
valuated matroid in terms of its valuated circuits, and then showing that the
corresponding valuation agrees with the function~\eqref{eq:valuation}.
Throughout the rest of the paper, we will use $F$ to denote the Frobenius
morphism $x \mapsto x^p$.

Recall that a (non-valuated) circuit of the algebraic matroid of the elements
$x_1, \ldots, x_n$ in the extension $L \supset K$ is an inclusion-wise minimal
set $C \subset E$ such that $K(x_{C})$ has transcendence degree $\lvert C \rvert
- 1$ over $K$. Therefore, there is a unique (up to scaling) polynomial
relation among the~$x_i$, which we call the \defi{circuit polynomial},
following~\cite{kiraly-rosen-theran}. More precisely, we let $K[X_C]$ be the
polynomial ring whose variables are denoted $X_i$ for $i \in C$. The
aforementioned circuit polynomial is a (unique up to scaling) generator $f_C$ of the kernel of the
homomorphism $K[X_C] \rightarrow K(x_C)$ which sends $X_i$ to $x_i$. We write
this polynomial:
\begin{equation*}
f_{C} = \sum_{\bfu \in J} c_{\bfu} X^{\bfu} \in K[X_C] \subset K[X_E]
\end{equation*}
where $J \subset \ZZ_{\geq 0}^n$ is a finite set of exponents and $c_{\bfu} \neq 0$
for all $\bfu \in J$. Then, we
define $\vecC(f_{C})$ to be the vector in $(\ZZ \cup \{\infty\})^n$ with
components:
\begin{equation}\label{eq:c-function}
\vecC(f_{C})_i =
\min \{\val_p \bfu_i \mid \bfu \in J, \bfu_i \neq 0 \},
\end{equation}
where $\val_p \bfu_i$ denotes the $p$-adic valuation, which is defined to be the
power of~$p$ in the prime factorization of the positive integer $\bfu_i$. If
$\bfu_i = 0$
for all $\bfu \in J$, then we take $\vecC(f_C)_i$ to be $\infty$.
For any vector $\vecC \in (\ZZ \cup
\{\infty\})^n$, the \defi{support} of~$\vecC$, denoted $\supp \vecC$, is the set
$\{i \in E \mid \vecC_i < \infty\}$.  Since $f_C$ is a polynomial in the
variables~$X_i$
for $i \in C$, but not in any proper subset of them, the support of $\vecC(f_C)$
is exactly the circuit~$C$.

We will take the valuated circuits of the Lindstr\"om valuation to be the set of
vectors:
\begin{equation}\label{eq:valuated-circuits}
\mathcal C = \{\vecC(f_C) + \lambda \onevector \mid \mbox{$C$ is a circuit of $L
\supset K$},
\lambda \in \ZZ\}
\subset (\ZZ \cup \{\infty\})^n,
\end{equation}
where $\onevector$ denotes the vector $(1, \ldots, 1)$.
Before verifying that this collection of vectors satisfies the axioms, we prove
the following preliminary lemma relating the
definition in~\eqref{eq:c-function} to the inseparable
degree:
\begin{lem}\label{l:poly-insep-degree}
Let $S \subset E$ be a set of rank $\lvert S \rvert - 1$, and let $C$ be the
unique circuit contained in $S$. If we abbreviate the vector $\vecC(f_C)$ as
$\vecC$, then
\begin{equation*}
[K(x_{S}) : K(x_{S \setminus \{i\}})]_i = p^{\vecC(f)_i}
\end{equation*}
for any $i \in C$.
In particular, $K(x_{S})$ is a separable extension of $K(x_{S \setminus
\{i\}})$ if and only if $\vecC_i = 0$.
\end{lem}

\begin{proof}
For $i \in C$, we let $Y_i$ denote the monomial $X_i^{p^{\vecC_i}}$ in $K[X_S]$.
Then, the polynomial~$f_C$ lies in the polynomial subring $K[X_{S \setminus
\{i\}}, Y_i]$, by the definition of~$\vecC_i$. Similarly, we let $y_i$ denote
the element $x_i^{p^{\vecC_i}} = F^{\vecC_i} x_i$ in $K(x_S)$. Then, $f_C$, as a
polynomial in $K[X_{S \setminus \{i\}}, Y_i]$, is the minimal defining relation
for $K(x_{S\setminus \{i\}}, y_i)$ as an extension of $K(x_{S \setminus
\{i\}})$. By the definition of $\vecC_i$, some term of $f_C$ is of the form
$X^uY_i^a$, where $a$ is not divisible by $p$, and so $\partial f_C/\partial
Y_i$ is a non-zero polynomial. Therefore, $f_C$ is a separable polynomial of
$Y_i$, and so $K(x_{S \setminus \{i\}}, y_i)$ is a separable extension of
$K(x_{S \setminus \{i\}})$.

On the other hand, $K(x_S)$ is a purely inseparable extension of $K(x_{S
\setminus \{i\}}, y_i)$, defined by the minimal relation
\begin{equation*}
x_i^{p^{\vecC_i}} - y_i = 0.
\end{equation*}
Therefore, this extension has degree $p^{\vecC_i}$, which is thus the inseparable
degree $[K(x_S) : K(x_{S \setminus \{x_i\}})]_i$, as desired.
\end{proof}

We now verify that the collection~\eqref{eq:valuated-circuits} satisfies the
axioms of valuated circuits. Several equivalent characterizations of valuated
circuits are given in~\cite{murota-tamura}, and we will use the characterization
in the following proposition:

\begin{prop}[Thm.~3.2 in \cite{murota-tamura}]\label{p:valuated-circuits}
A set of vectors $\mathcal C \subset (\ZZ \cup \{\infty\})^n$ is the set of
\defi{valuated circuits} of a valuated matroid if and only if it satisfies the
following properties:
\begin{enumerate}
\item
The collection of sets $\{\supp \vecC \mid \vecC \in \mathcal C\}$ satisfies
the axioms of the circuits of a non-valuated matroid.
\item
If $\vecC$ is a valuated circuit, then $\vecC + \lambda \onevector$ is a
valuated circuit for all $\lambda \in \ZZ$.
\item
Conversely, if $\vecC$ and $\vecC'$ are valuated circuits with $\supp \vecC =
\supp \vecC'$, then $\vecC = \vecC' + \lambda \onevector$ for some integer
$\lambda$.
\item
Suppose $\vecC$ and $\vecC'$ are in $\mathcal C$ such that
\begin{equation*}
\rank(\supp \vecC \cup \supp \vecC') = \lvert \supp \vecC \cup \supp \vecC'
\rvert - 2,
\end{equation*}
and $u, v \in E$ are elements such that $\vecC_u = \vecC_u'$ and $\vecC_v <
\vecC_v' = \infty$. Then there exists a vector $\vecC'' \in \mathcal C$ such
that $\vecC_u'' = \infty$, $\vecC_v'' = \vecC_v$, and $\vecC_i'' \geq \min\{
\vecC_i, \vecC_i'\}$ for all $i \in E$ .
\end{enumerate}
\end{prop}

The first property from Proposition~\ref{p:valuated-circuits} is
equivalent to axioms $\mathrm{VC1}$, $\mathrm{VC2}$, and $\mathrm{MCE}$
from~\cite{murota-tamura} and the three after that are denoted $\mathrm{VC3}$,
$\mathrm{VC3_e}$, $\mathrm{VCE_{loc1}}$, respectively.

\begin{prop}\label{p:circuit-axioms}
The collection~$\mathcal C$ of vectors given in~\eqref{eq:valuated-circuits} 
defines the valuated circuits of a valuated matroid.
\end{prop}

\begin{proof}
The first axiom from Proposition~\ref{p:valuated-circuits} follows because each
valuated circuit is constructed to have support equal to a non-valuated circuit.
The second axiom follows immediately from the construction, and the third
follows from the uniqueness of circuit polynomials.

Thus, it remains only to check (4) from Proposition~\ref{p:valuated-circuits}.
Suppose that $\vecC$ and~$\vecC'$ are valuated circuits and $u, v \in E$ are
elements satisfying the hypotheses of condition~(4). We can write $\vecC =
\vecC(f) + \lambda \onevector$ and $\vecC' = \vecC(f') + \lambda'\onevector$ for
circuit polynomials~$f$ and~$f'$ in $K[X_1, \ldots, X_n]$. Note that $\vecC(F^m
f) = \vecC(f) + m \onevector$, and so by either replacing $f$ with $F^m f$ or 
replacing $f'$ with $F^m f'$, for some integer~$m$,
we can assume that $\lambda = \lambda'$. Moreover,
since the fourth axiom only depends on the relative values of the entries of
$\vecC$ and $\vecC'$, it is sufficient to check the axiom for $\vecC$
and~$\vecC'$ replaced by $\vecC(f) = \vecC - \lambda \onevector$ and $\vecC(f')
= \vecC' - \lambda \onevector$, respectively.

We now define an injective homomorphism $\psi$ from
the polynomial ring $K[Y_1, \ldots, Y_n]$ to $K[X_1, \ldots, X_n]$ by
\begin{equation*}
\psi(Y_i) = F^{\min\{\vecC_i, \vecC_i'\}} X_i
\end{equation*}
Thus, there exist polynomials $g$ and $g'$ in $K[Y_1, \ldots
Y_n]$ such that $f = \psi(g)$ and $f' =
\psi(g')$. In particular, since $\vecC(g)_i = \vecC_i -
\min\{\vecC_i, \vecC_i'\}$ and $\vecC(g')_i = \vecC_i' - \min\{\vecC_i,
\vecC_i'\}$, then our assumptions on $u$ and $v$ imply that $\vecC(g)_u =
\vecC(g')_u = \vecC(g)_v = 0$.

Likewise, we define $y_i = F^{\min\{\vecC_i, \vecC_i'\}} x_i$ so that the
elements $y_i \in L$
satisfy the polynomials $g$ and $g'$. Thus,
Lemma~\ref{l:poly-insep-degree} shows that $g$ is separable in the variable
$Y_v$, and so if $S$ denotes the set $\supp \vecC \cup \supp \vecC'$, then
$K(y_S)$ is a separable extension of $K(y_{S \setminus \{v \}})$. Likewise,
$g'$ is separable in the variable $Y_u$ and doesn't use the variable
$Y_v$, and so $K(y_{S \setminus \{v\}})$ is a separable extension of $K(y_{S
\setminus \{v, u\}})$. Since the composition of separable extensions is
separable, $y_v$ is separable over $K(y_{S \setminus
\{v,u\}})$~\cite{lang}*{Thm.~V.4.5}.

Since algebraic extensions have transcendence degree $0$, then the field $K(y_{S
\setminus \{v, u\}})$ has the same transcendence degree over $K$ as $K(y_S)$ does,
and that transcendence degree is $\lvert S
\rvert - 2$, because we assumed that $\rank(S) = \lvert S \rvert - 2$. In
addition, we
have containments
$K(y_{S \setminus \{v,u\}}) \subset K(y_{S \setminus\{u\}}) \subset K(y_{S})$,
so that $K(y_{S \setminus \{u\}})$ also has transcendence degree $\lvert S \rvert
-2$, and therefore there exists a unique (up to scaling) polynomial relation
$g'' \in K[Y_{S \setminus \{u\}}]$ among the elements $y_i$ for $i \in S
\setminus \{u\}$. Since $y_v$ is finite and separable over $K(y_{S \setminus
\{u\}})$, $\vecC(g'')_v = 0$ by Lemma~\ref{l:poly-insep-degree}.

We claim that the $\vecC''= \vecC(\psi(g''))$ satisfies the desired conclusions
of the axiom.
First,
\begin{equation*}
\vecC_v'' = \vecC(g_{C''})_v + \min\{\vecC_v, \vecC_v'\} = 0 + \min\{\vecC_v,
\infty\} = \vecC_v,
\end{equation*}
as desired. Similarly,
\begin{equation*}
\vecC''_i = \vecC (g_{C''})_i + \min \{\vecC_i, \vecC_i'\} \geq
\min \{\vecC_i, \vecC_i'\},
\end{equation*}
and, finally, $\vecC''_u = \infty$ because $g''$ was chosen to be a polynomial
in the variables $Y_{S \setminus \{u\}}$.
\end{proof}

\begin{rmk}\label{r:computing}
The valuated circuits defined in Proposition~\ref{p:circuit-axioms} are
effectively computable from a suitable description of $L$ and the $x_i$. More
precisely, suppose $K$ is a finitely generated extension of $\mathbb F_p$ and
$L$ is given as the fraction field of $K[x_1,\ldots, x_n]/I$ for a prime
ideal~$I$.
Then $I$ can be represented in computer algebra software, and the elimination
ideals $I \cap K[x_S]$ can be computed for any subset $S \subset E$ using
Gr\"obner basis methods. The circuits of the algebraic matroid are the minimal
subsets $C$ for which $I \cap K[x_C]$ is not the zero ideal, in which case the
elimination ideal will be principal, generated by the circuit polynomial~$f_C$. By computing all of these elimination
ideals, we can determine the circuits of the algebraic matroid, and from the
corresponding generators, we get the valuated circuits
by the formula~\eqref{eq:valuated-circuits}.
\end{rmk}

\begin{ex}\label{ex:toric}
One case where the connection between the Lindstr\"om valuated matroid and linear
algebraic valuated matroids is most transparent is when the variables $x_i$ are
monomials. This example is given in~\cite{bollen-draisma-pendavingh}*{Thm.~45},
but we discuss it here in terms of our description of the valuated circuits.

We let $A$ be any $d \times n$ integer matrix, and then we take $L = K(t_1,
\ldots, t_d)$ for any field $K$ of characteristic~$p$, and we let $x_i$ be the
monomial $t_1^{A_{1i}} \cdots t_d^{A_{di}}$, whose exponents are the $i$th
column of $A$. Then the algebraic matroid of $x_1, \ldots, x_n$ is the same as
the linear matroid of the vector configuration formed by taking the columns
of~$A$. Moreover, we claim that the Lindstr\"om valuated matroid is the same as
the valuated matroid of the same vector configuration with respect to the
$p$-adic valuation on~$\mathbb Q$.

To see this, we look at the valuated circuits of both valuated matroids. A
circuit of the linear matroid is determined by an $n \times 1$ vector $\bfu$ with
minimal support such that $A \bfu = \mathbf 0$. The circuit is the support of
the vector $\bfu$, and the valuated circuit is the entry-wise $p$-adic valuation
of~$\bfu$. The support of~$\bfu$ is also a circuit of $x_1, \ldots, x_n$ with
circuit polynomial
\begin{equation*}
f = X_1^{\bfu^{(+)}_1} \cdots X_n^{\bfu^{(+)}_n} -
X_1^{\bfu^{(-)}_1} \cdots X_n^{\bfu^{(-)}_n}
\end{equation*}
where
\begin{equation*}
\bfu^{(+)}_i = \min\{0, \bfu_i\} \qquad \bfu^{(-)}_i = -\max\{0, \bfu_i\}
\end{equation*}
so that $\bfu = \bfu^{(+)} - \bfu^{(-)}$. Then, since one of
$\val_p(\bfu_i^{(-)})$ and $\val_p(\bfu_i^{(+)})$ equals $\val_p(\bfu_i)$ and
the other is infinite, $\vecC(f)$ is the same as the entry-wise $p$-adic valuation
of $\bfu$, which is the valuated circuit of the linear matroid. Thus, the
valuated circuits of the linear and algebraic matroids are the same.
\end{ex}

\begin{prop}\label{p:valuation-circuits}
The Lindstr\"om valuated matroid given by the circuits
in~\eqref{eq:valuated-circuits} agrees with the valuation~\eqref{eq:valuation}
given in the introduction.
\end{prop}

\begin{proof}
The essential relation between the valuation and the valuated circuits is that
if $B$ is a basis, $u \in B$, $v \in E \setminus B$, and $\vecC$ is a valuated
circuit whose support is contained in $B \cup \{v\}$, then:
\begin{equation}\label{eq:valuated-basis-circuit}
\nu(B) + \vecC_u = \nu(B \setminus \{u\} \cup \{v\}) + \vecC_v
\end{equation}
This relation is used at the beginning of \cite{murota-tamura}*{Sec. 3.1} to
define the valuated circuits in terms of the valuation, and in the other
direction with~(10) from~\cite{murota-tamura}.
In~\eqref{eq:valuated-basis-circuit}, we adopt the convention that $\nu(B
\setminus \{u\} \cup \{v\})$ is $\infty$ if $B \setminus \{u\} \cup \{v \}$ is
not a basis.

The only quantities in \eqref{eq:valuated-basis-circuit} which can
be infinite are $\vecC_u$ and $\nu(B \setminus \{u\} \cup \{v\})$, because if $\vecC_v$
were infinite, then $\supp \vecC$ would be contained in $B$, which contradicts $B$
being a basis. However $B \setminus \{u\} \cup \{v\}$ is not a basis if and only
the support of $\vecC$ is contained in $B \setminus \{u \} \cup \{v\}$, which is
true if and only $\vecC_u = \infty$. Therefore, the left hand side
of~\eqref{eq:valuated-basis-circuit} is infinite if and only if the right hand
side is, so for the rest of the proof, we can assume that all of the terms
of~\eqref{eq:valuated-basis-circuit}
are finite.

By the multiplicativity of inseparable degrees~\cite{lang}*{Cor.~V.6.4}, we have 
\begin{align*}
\nu(B) &= \log_p [L : K(x_B)]_i \\
&= \log_p [L : K(x_{B \cup \{v\}})]_i +
\log_p [K(x_{B \cup \{v\}}) : K(x_B)]_i \\
&= \log_p [L : K(x_{B \cup \{v\}})]_i + {\vecC_v},
\end{align*}
by Lemma~\ref{l:poly-insep-degree}. Similarly, we also have
\begin{align*}
\nu(B \setminus \{u\} \cup \{v\}) &=
\log_p [L : K(x_{B \setminus \{u\} \cup \{v\}}]_i \\
&= \log_p [L : K(x_{B \cup \{v\}})]_i 
+ \log_p [K(x_{B \cup \{v\}}) : K(x_{B \setminus \{u\} \cup \{v\}})]_i \\
&= \log_p [L : K(x_{B \cup \{v\}})]_i + \vecC_u,
\end{align*}
again, using Lemma~\ref{l:poly-insep-degree} for the last step.
Therefore,
\begin{equation*}
\nu(B) - \vecC_v = \log_p [L : K(x_{B \cup \{v\}})]_i = 
\nu(B \setminus \{u\} \cup \{v \}) - \vecC_u,
\end{equation*}
which is just a rearrangement of the desired equation
\eqref{eq:valuated-basis-circuit}.
\end{proof}

Thus, we've proved the first part of Theorem~\ref{t:agree}, namely that the
function~$\nu$ given in~\eqref{eq:valuation} defines a valuation on the
algebraic matroid $M$. In the next section, we turn to the second part of
Theorem~\ref{t:agree} and show that this valuation is compatible with the
matroid flock studied in~\cite{bollen-draisma-pendavingh}.

\section{Matroid flocks}

We now show that the matroid flock defined by the valuated matroid from the
previous section is the same as the matroid flock defined from the extension $L
\supset K$ in~\cite{bollen-draisma-pendavingh}. A \defi{matroid flock} is a
function~$M$ which maps each vector $\alpha \in \ZZ^n$ to a matroid
$M_\alpha$ on the set $E$, such that:
\begin{enumerate}
\item $M_\alpha \slash i = M_{\alpha + e_i} \backslash i$ for all $\alpha \in
\ZZ^n$ and $i \in E$,
\item $M_\alpha = M_{\alpha + \onevector}$ for all $\alpha \in \ZZ^n$.
\end{enumerate}
In the first axiom, the matroids $M_\alpha \slash i$ and $M_{\alpha + e_i} \backslash i$
are the contraction and deletion of the respective matroids with respect to the
single element~$i$.

To any valuated matroid $M$, the associated
matroid flock, which we also denote by $M$, is defined by letting $M_\alpha$ be
the matroid whose bases
consist of those bases of $M$ such that
$\mathbf e_B \cdot \alpha - \nu(B) = g(\alpha)$,
where $\mathbf e_B$ is the indicator vector with entry $(\mathbf e_B)_i = 1$ for
$i \in B$ and $(\mathbf e_B)_i = 0$
otherwise, and where
\begin{equation}\label{eq:def-g}
g(\alpha) = \max\{ \mathbf e_B \cdot \alpha - \nu(B) \mid B \mbox{ is a basis of
$M$}\}.
\end{equation}
Moreover, any matroid flock comes from a valuated matroid in this way by
Theorem~7 in~\cite{bollen-draisma-pendavingh}.

On the other hand, \cite{bollen-draisma-pendavingh} also associates a matroid
flock directly to the extension $L \supset K$ and the elements $x_1, \ldots,
x_n$. Their construction is in terms of algebraic varieties and the tangent
spaces at sufficiently general points. Here, we recast their definition using
the language of field theory and derivations. Define $\tilde L$ to be the
perfect closure of $L$, which is equal to the union
$\bigcup_{k \geq 0} L(x^{1/p^k}_1, \ldots, x_n^{1/p^k})$
of the infinite tower of
purely inseparable extensions of~$L$.
For a vector $\alpha \in \ZZ^n$, we define $F^{-\alpha} x_E$ to be the vector in
$\tilde L^n$ with $(F^{-\alpha} x_E)_i = F^{-\alpha_i} x_i$, and $K(F^{-\alpha}
x_E)$ to be the field generated by these elements. Recall from field theory,
e.g.\ \cite{lang}*{Sec. XIX.3}, that the vector space of differentials
$\Omega_{K(F^{-\alpha}x_E)/ K}$ is defined algebraically over
$K(F^{-\alpha}x_E)$, generated by the differentials $d (F^{-\alpha_i} x_i)$ as
$i$ ranges over the set~$E$. We define $N_\alpha$ to be the matroid on $E$ of
the configuration of these vectors $d(F^{-\alpha_i} x_i)$ in
$\Omega_{K(F^{-\alpha}x_E)/K}$, and then the function $N$ which sends $\alpha$
to $N_\alpha$ is a matroid flock~\cite{bollen-draisma-pendavingh}*{Thm.~34}.

\begin{proof}[Proof of Theorem~\ref{t:agree}]
The function $\nu$ is a valuation on $M$ by
Propositions~\ref{p:circuit-axioms} and~\ref{p:valuation-circuits}, so it only
remains to show that the matroid flock associated to this valuation
coincides with the matroid flock $N$ defined above.
Let $\alpha$ be a vector in $\ZZ^n$. Since both $M$ and $N$ are
matroid flocks, they are invariant under shifting $\alpha$ by the vector
$\onevector$, as in the second axiom of a matroid flock.
Therefore, we can shift
$\alpha$ by a multiple of $\onevector$ such that all entries of $\alpha$ are
non-negative and it suffices to show $M_\alpha = N_\alpha$ in this case.

Now let $B$ be a basis of $M$ and we want to show that the differentials
$d(F^{-\alpha_i}x_i)$, for $i \in B$, form a basis for $\Omega_{K(F^{-\alpha}
x_E) / K}$ if and only if ${\mathbf e_B \cdot \alpha} - \nu(B)$ equals 
$g(\alpha)$, as defined in~\eqref{eq:def-g}. Since the field $K(F^{-\alpha}
x_B)$ is generated by the algebraically independent elements $F^{-\alpha_i} x_i$
as $i$ ranges over the elements of $B$, the differentials $d(F^{-\alpha_i}x_i)$ do form a basis for
$\Omega_{K(F^{-\alpha} x_B)/K}$. Moreover, the natural map
$\Omega_{K(F^{-\alpha} x_B)/K} \rightarrow \Omega_{K(F^{-\alpha} x_E)/K}$ is an
isomorphism if and only if the $K(F^{-\alpha} x_E)$ is a separable extension of
$K(F^{-\alpha} x_B)$~\cite{lang}*{Prop.~VIII.5.2}, i.e.\ if and only if its
inseparable degree is~$1$. Therefore, $B$ is a basis for $N_\alpha$ if and only
if $[K(F^{-\alpha} x_E) : K(F^{-\alpha} x_B)]_i = 1$.

We list the inseparable degrees:
\begin{align*}
[L : K(x_B)]_i &= p^{\nu(B)} \\
[K(F^{-\alpha} x_B) : K(x_B)]_i &= p^{\mathbf e_B \cdot \alpha} \\
[K(F^{-\alpha} x_E) : K(x_E)]_i &= p^{m(\alpha)} \\
[L : K(x_E)]_i &= p^\ell
\end{align*}
The first of these equalities is by definition, the second is because
$K(F^{-\alpha} x_B)$
is the purely inseparable extension of $K(x_B)$ defined by adjoining
a $p^{\alpha_i}$-root of $x_i$ for each~$i$, and the third and fourth we take to
be the
definitions of the integers $m(\alpha)$ and~$\ell$, respectively.
By the
multiplicativity of inseparable degrees, and taking logarithms, we have:
\begin{align}\label{eq:insep-degree}
\log_p [K(F^{-\alpha} x_E) : K(F^{-\alpha} x_B)]_i &=
\log_p [K(F^{-\alpha} x_E) : K(x_B)]_i - \mathbf e_B \cdot \alpha \notag \\
&= m(\alpha) + [K(x_E) : K(x_B)]_i - \mathbf e_B \cdot \alpha \notag \\
&= m(\alpha) - \ell + \nu(B) - \mathbf e_B \cdot \alpha
\end{align}
As noted above, $B$ is a basis of $N_\alpha$ if and only if the left hand
side of \eqref{eq:insep-degree} is zero, and $B$ is a basis of $M_{\alpha}$ if
and only $\mathbf e_B \cdot \alpha - \nu(B) = g(\alpha)$.
Thus, it suffices to show that $m(\alpha) - \ell$ equals $g(\alpha)$.

Since \eqref{eq:insep-degree} is always non-negative, we have the inequality
\begin{equation*}
m(\alpha) - \ell \geq {\mathbf e_B \cdot \alpha} - \nu(B)
\end{equation*}
for all bases $B$,
and thus $m(\alpha) - \ell \geq g(\alpha)$. On the other hand, if $m(\alpha) -
\ell > g(\alpha)$, then
\eqref{eq:insep-degree} will always be positive, so no subset of the
differentials $d(F^{\alpha_i}x_i)$ will form a basis for $\Omega_{K(F^{-\alpha}
x_E)/K}$. However, this would contradict the fact that the complete set of
differentials $d(F^{-\alpha_i}x_i)$ for all $i \in E$ forms a generating set for
$\Omega_{K(F^{-\alpha} x_E)/K}$, and therefore, some subset forms a basis. Thus,
$m(\alpha)$ must equal $g(\alpha) + l$, which completes the proof that the two
matroid flocks coincide.
\end{proof}

\begin{rmk}\label{r:equivalence}
By \cite{bollen-draisma-pendavingh}*{Thm.~7}, any matroid flock, such as that of
an algebraic extension, comes from a valuated matroid, but the valuation is not
unique. In particular, two valuations $\nu$ and $\nu'$ are called
\defi{equivalent} if they differ by a shift $\nu'(B) = \nu(B) + \lambda$ for
some constant $\lambda$~\cite{dress-wenzel}*{Def.~1.1}, and equivalent valuations define the same matroid
flock. However, among all equivalent valuations giving the matroid flock of an
algebraic extension, the formula~\eqref{eq:valuation} nevertheless gives a
distinguished valuation. For example, if $L = K(x_E)$, then this distinguished
valuation $\nu$ is the unique representative such that the minimum $\min_B
\nu(B)$ over all bases $B$ is $0$. If $L$ is a proper extension of $K(x_E)$,
then the valuation $\nu$ records the inseparable degree $[L: K(x_E)]_i$, which
was denoted $p^\ell$ in the proof of Theorem~\ref{t:agree}.
\end{rmk}

\begin{ex}\label{ex:non-fano}
We look at the matroid flock and Lindstr\"om valuation of an algebraic
realization of the non-Fano matroid~$M$ over $K = \mathbb F_2$, which is a
special case of the construction
in Example~\ref{ex:toric}. The realization is given by the elements
\begin{align*}
x_1 &= t_1 &
x_3 &= t_3 &
x_5 &= t_1 t_3 &
x_7 &= t_1t_2t_3 \\
x_2 &= t_2 &
x_4 &= t_1 t_2 &
x_6 &= t_2 t_3
\end{align*}
in the field $L = K(t_1, t_2, t_3)$. The differentials of these
elements in $\Omega_{L/K}$ are:
\begin{align*}
dx_1 &= dt_1 &
dx_4 &= t_2 \,dt_1 + t_1\, dt_2 \\
dx_2 &= dt_2 &
dx_5 &= t_3\, dt_1 + t_1 \,dt_3 \\
dx_3 &= dt_3 &
dx_6 &= t_3 \, dt_2 + t_2 \,dt_3 \\
&& dx_7 &= t_2 t_3 \,dt_1 + t_1 t_3 \,dt_2 + t_1 t_2 \,dt_3.
\end{align*}
These vectors are projectively equivalent to the Fano configuration, and,
therefore, the matroid $M_{(0,0,0,0,0,0,0)}$ of the matroid flock is the Fano
matroid. In particular, we have the linear relation $t_3 \, dx_4 + t_2 \,dx_5 +
t_1 \, dx_6 = 0$, among the differentials, even though $\{4, 5, 6\}$ is a basis
of the algebraic matroid.

On the other hand, if we let $\alpha = (-1, -1, -1, 0, 0, 0, -1)$, then
$K(F^{-\alpha}x_E)$ is the subfield $K(x_4, x_5, x_6) \subset L$, because
\begin{align*}
Fx_1 = x_1^2 &= x_4 x_5 x_6^{-1} &
Fx_3 = x_3^2 &= x_4^{-1} x_5 x_6 \\
Fx_2 = x_2^2 &= x_4 x_5^{-1} x_6 &
Fx_7 = x_7^2 &= x_4 x_5 x_6 
\end{align*}
Therefore, $\{4, 5, 6\}$ is a basis for the matroid $M_\alpha$. 
Using the basis $dx_4, dx_5, dx_6$ for $\Omega_{K(F^{-\alpha}x_E)/K}$, one can
check that the vectors $d(Fx_i)$, for $i = 1, 2, 3, 7$ are all parallel to each
other, and thus the bases of $M$ which contain at least two of these indices is
not a basis for $M_{\alpha}$.

We claim that the Lindstr\"om valuation~$\nu$ of the field extension $L$ of $K$
takes the value $0$ for every basis of $M$ except that $\nu(\{4, 5, 6\}) = 1$.
This can be seen directly from the definition~\eqref{eq:valuation} because one
can check that every basis other than $\{4, 5, 6\}$ generates the field $L$, and
$L \supset K(x_4, x_5, x_6)$ is an index~$2$, purely inseparable extension.

Alternatively, the fact that the vector configuration of the differentials
$dx_i$ in $\Omega_{L/K}$ is the Fano matroid means that its bases consist of all
bases of $M$ except for $\{4, 5, 6\}$, and so the bases of the Fano matroid have
the same valuation, except for $\{4,5,6\}$, which has larger valuation. As in
Remark~\ref{r:equivalence}, the matroid flock only determines the valuation up
to equivalence, so we can take $\nu(B) = 0$ for $B$ a basis of the Fano
matroid. Then, the computation of $M_{\alpha}$ above shows that both $\{4, 5,
6\}$ and $\{3, 5, 6\}$ are bases, and thus,
\begin{equation*}
\mathbf e_{\{4, 5, 6\}}
 \cdot \alpha - \nu(\{4, 5, 6\}) =
\mathbf e_{\{3,5,6\}} \cdot \alpha - \nu(\{3, 5, 6\}) = -1 - 0 = -1
\end{equation*}
and so we can solve for $\nu(\{4, 5, 6\}) = 1$.

Finally, a third way of computing the Lindstr\"om valuation is to use
Example~\ref{ex:toric}, which shows that the valuation is the same as that of
the vector configuration given by the columns of the matrix
\begin{equation*}
A = 
\begin{pmatrix}
1 & 0 & 0 & 1 & 1 & 0 & 1 \\
0 & 1 & 0 & 1 & 0 & 1 & 1 \\
0 & 0 & 1 & 0 & 1 & 1 & 1
\end{pmatrix}
\end{equation*}
over the field of rational numbers $\QQ$ with the $2$-adic valuation. The
valuation of a vector configuration is given by the $2$-adic valuation of the
determinant of the submatrices. The submatrices of $A$ corresponding to bases
of~$M$ all have determinant $\pm 1$ except for the one with
columns $\{4,
5, 6\}$, whose determinant is $-2$, which has $2$-adic valuation equal to $1$.
\end{ex}

\section{Cocircuits and minors}\label{sec:cocircuits-minors}

In this section, we consider further properties of the Lindstr\"om valuated
matroid which can be understood in terms of the field theory of the extension.
In particular, we give constructions of the valuated cocircuits and minors of
the Lindstr\"om valuated matroid.

First, a \defi{hyperplane} of the algebraic matroid of $L$ is a maximal subset
$H$ of~$E$ such that $L$ has transcendence degree $1$ over $K(x_E)$. For any
hyperplane~$H$, we define a vector in $(\ZZ \cup
\{\infty\})^n$:
\begin{equation*}
\vecCco(H)_i = \begin{cases}
\infty &\mbox{if } i \in H \\
\log_p [L : K(x_{H \cup \{i\}})]_i & \mbox{if } i \notin H
\end{cases}
\end{equation*}
The expression in the second case is an integer by \cite{lang}*{Cor.~V.6.2} and
finite because, by the assumption that $H$ is a hyperplane, $L$ must be an
algebraic extension of $K(x_{H \cup \{i\}})$, for $i \notin H$.

\begin{prop}
The collection of vectors:
\begin{equation*}
\{\vecCco(H) + \lambda \onevector \mid
H \mbox{ is a hyperplane of the algebraic matroid of } L, \lambda \in \ZZ \}
\end{equation*}
define the cocircuits of the Lindstr\"om valuation of the field $L$ and the
elements $x_1, \ldots, x_n$.
\end{prop}

\begin{proof}
By definition, the cocircuits of a valuated matroid~$M$ are the circuits of the
dual $M^*$, and the dual valuation is defined by $\nu(B^*) = \nu(E \setminus
B^*)$ for any subset $B^* \subset E$ such that $E \setminus B^*$ is a basis
of~$M$. Suppose $B^*$ and $B^* \setminus \{u\} \cup \{v\}$ are bases of $M^*$,
and $\vecCco(H)$ is a cocircuit contained in $B^* \cup \{v\}$. Then, as in the
proof of Proposition~\ref{p:valuation-circuits}, we have to show the relation:
\begin{equation}\label{eq:cocircuit}
\nu^*(B^*) + \vecCco(H)_u = \nu^*(B^* \setminus \{u\} \cup \{v\}) + \vecCco(H)_v
\end{equation}
We write $B$ for the complement $E \setminus B^*$, which is a basis of~$M$.
We can then expand these expressions using their definitions and
multiplicativity of the inseparable degree:
\begin{align*}
\nu^*(B^*) &= \log_p [L : K(x_{H \cup \{v\}})]_i
+ \log_p [K(x_{H \cup \{v\}}) : K(x_{B})]_i \\
\vecCco(H)_u &= \log_p [L : K(x_{H \cup \{u\}})]_i \\
\nu^*(B^* \setminus \{u\} \cup \{v\}) &=
\log_p [L : K(x_{H \cup \{u\}})]_i \\
&\qquad\qquad + \log_p [K(x_{H \cup \{u\}}) :
    K(x_{B \setminus \{v\} \cup \{u\}})]_i \\
\vecCco(H)_v &= \log_p [L : K(x_{H \cup \{v\}})]_i
\end{align*}
Therefore, to show the relation \eqref{eq:cocircuit}, it is sufficient to show
that 
\begin{equation}\label{eq:cocircuit-rewritten}
[K(x_{H \cup \{v\}}) : K(x_{B})]_i =
[K(x_{H \cup \{u\}}) : K(x_{B \setminus \{v\} \cup \{u\}})]_i
\end{equation}

We claim that \eqref{eq:cocircuit-rewritten} is true because both sides are
equal to the inseparable degree $[K(x_{H}) : K(x_{B \setminus \{v\}})]_i$.
Indeed, the extensions on either side of \eqref{eq:cocircuit-rewritten} are
given adjoining to the extension $K(x_H) \supset K(X_{B \setminus \{v\}})$ a
single transcendental element, namely, $x_v$ on the left, and $x_u$ on the
right. Such a transcendental element has no relations with the other elements of
$x_H$ and so doesn't affect the inseparable degree.
\end{proof}

Minors of a valuated matroid are defined in~\cite{dress-wenzel}*{Prop. 1.2
and~1.3}. Note that the definition of the valuation on the minor depends on an
auxiliary choice of a set of vectors, and the valuation is only defined up to
equivalence.

\begin{prop}
Let $F$ and $G$ be disjoint subsets of $E$. Then the minor $M \backslash G / F$,
denoting the deletion of $G$ and the contraction of~$F$,
is equivalent to the Lindstr\"om valuation of the extension $K(x_{E \setminus G}) \supset
K(x_{F})$ with the elements $x_i$ for $i \in E \setminus (F \cup G)$.
\end{prop}

\begin{proof}
The valuated circuits of the deletion $M \backslash G$ are equal to the restriction of
the valuated circuits $\vecC$ such that $\supp \vecC \cap G = \emptyset$ to the
indices $E \setminus G$. Likewise, the circuits and circuit polynomials of the
algebraic extension
$K(x_{E'}) \supset K$ are those of $L \supset K$ such that the support and
variable indices, respectively, are disjoint from~$G$. Therefore, the valuated
circuits of the Lindstr\"om matroid of $K(x_{E \setminus G})$ as an extension of $K$ are
the same as those of the deletion~$M \backslash G$.

Dually, the valuated cocircuits of the contraction $M \backslash G / F$ are
the restrictions of the cocircuits $\vecCco$ of $M \backslash G$ such that $\supp
\vecCco \cap F = \emptyset$ to the indices in $E \setminus (F \cup G)$. The hyperplanes of
the extension $K(x_{E \setminus G}) \supset K(x_F)$ are the hyperplanes of
$K(x_{E \setminus G}) \supset K$ which contain $F$ and so the valuated cocircuits are the
valuated cocircuits which are disjoint from $F$ and with indices restricted to
the indices $E \setminus (
F \cup G)$. Therefore, the Lindstr\"om valuated matroid of $K(x_{E \setminus G}) \supset K(x_F)$ is
the same as the minor $M \backslash G / F$.
\end{proof}

\end{document}